\documentclass[11pt,a4paper,twoside,leqno]{article} 
\usepackage[english]{babel}

\usepackage[a4paper,top=3.5cm,bottom=3cm,left=3.5cm,right=3.5cm,bindingoffset=5mm]{geometry}
\usepackage{palatino}  
\usepackage{microtype}
\usepackage[sc]{mathpazo} 
\usepackage[applemac]{inputenc}
\usepackage{amsmath,amsthm,amssymb,amsfonts,mathrsfs,bm,shapepar}
\usepackage{braket}                      
\usepackage{enumitem}                
\usepackage{indentfirst}
\usepackage{tikz}
\usepackage{tikz-cd}
\usetikzlibrary{matrix,arrows,decorations.pathmorphing}

\usepackage{titlesec}
\titleformat*{\section}{\large\scshape}
\titleformat*{\subsection}{\normalsize\bfseries}
\titleformat*{\subsubsection}{\small\bfseries}

\newcommand{\ssection}[1]{\section[#1]{\centering #1}}         

\theoremstyle{definition}

\theoremstyle{plain}

\newtheoremstyle{rema} 
        {3mm}
        {3mm}
        {\rmfamily}
        {0pt}
        {\scshape}
        {.}
        {2mm}
        {}
\theoremstyle{rema}
\newtheorem{rem}{Remark}[section]

\newtheoremstyle{thm} 
        {4mm}
        {4mm}
        {\itshape}
        {0pt}
        {\scshape}
        {.}
        {2mm}
        {}
\theoremstyle{thm}
\newtheorem{teo}{Theorem}
\newtheorem{lemma}{Lemma}

\DeclareMathOperator{\Hilb}{Hilb}
\DeclareMathOperator{\Sym}{Sym}
\DeclareMathOperator{\mult}{mult}
\DeclareMathOperator{\Supp}{Supp}

\usepackage{fancyhdr}
\author{Martin G. Gulbrandsen \and Andrea T. Ricolfi}
\title{The Euler Characteristic of the Generalized Kummer Scheme}

\makeatletter
\newcommand\Authors{Martin G. Gulbrandsen, \and Andrea T. Ricolfi}
\let\Title\@title
\makeatother

\title{\normalsize{\textbf{THE EULER CHARATERISTIC OF THE GENERALIZED KUMMER SCHEME OF AN ABELIAN THREEFOLD}}}
\author{Martin G. Gulbrandsen \and Andrea T. Ricolfi}
\date{}

\begin{document}
\maketitle

\begin{abstract}
Let $X$ be an Abelian threefold. We prove a formula, conjectured by the first author,
expressing the Euler characteristic of the generalized Kummer schemes $K^nX$ of
$X$ in terms of the number of plane partitions. This computes the
Donaldson-Thomas invariant of the moduli stack $[K^nX/X_n]$.
\end{abstract}

\bigskip
\ssection{Introduction}

Let $n>0$ be an integer. The $n$-th
\emph{generalized Kummer scheme} $K^nX$ of an Abelian variety $X$ is the fibre
over $0_X$ of the composite map 
\[
\Hilb^nX\rightarrow \Sym^nX\rightarrow X,
\]
where the first arrow is the Hilbert-Chow morphism
and the second arrow takes a cycle to the weighted sum of its supporting points.
The purpose of this note is to prove the following formula, which is the three-dimensional
case of a conjecture from~\cite{G2}:

\begin{teo}\label{thm}
Let $X$ be an Abelian threefold. The Euler characteristic of its generalized Kummer Scheme $K^nX$ is
\begin{equation*}
\chi(K^nX)=n^5\sum_{d|n}d^2.
\end{equation*}
\end{teo}

Simultaneously with and independent of our work, Shen \cite{shen} has proven the conjecture
in \cite{G2} for $X$ an Abelian variety of arbitrary dimension $g$, stating that
\begin{equation}\label{eq:chi-series}
\sum_{n\geq 0}P_{g-1}(n)q^n=\exp\Bigl(\sum_{n\geq 1}\frac{\chi(K^nX)}{n^{2g}}q^n\Bigr),
\end{equation}
where $P_d(n)$ denotes the number of $d$-dimensional partitions of $n$.
In fact, Shen proves a further generalization of this to the case of a product $X\times Y$,
where one factor $X$ is an Abelian variety, and the other factor $Y$ is an
arbitrary quasi-projective variety.  For $g=3$, the formula in Theorem
\ref{thm} is recovered from \eqref{eq:chi-series} by applying MacMahon's product
formula for plane partitions (cf.~\cite[Cor. 7.20.3]{St}).

One motivation for the computation of $\chi(K^n X)$ is as a test case for
Donaldson--Thomas invariants for Abelian threefolds, as developed in~\cite{G2}.
In particular (see \emph{loc.\ cit.}), the Donaldson-Thomas invariant of the moduli stack $[K^nX/X_n]$ is the rational number
\[
\frac{(-1)^{n+1}}{n^6}\chi(K^nX)=\frac{(-1)^{n+1}}{n}\sum_{d|n}d^2.
\]

The formula \eqref{eq:chi-series} could be motivated by formally expanding Cheah's formula, for the Euler characteristic of Hilbert schemes of points (cf.~\cite{Ch}, and also~\cite{GLM} for a motivic refinement) up to first order in $\chi(X)$, as follows:
\begin{align}
1+\displaystyle\sum_{n\geq 1}&\chi(\Hilb^nX)q^n =
1+\chi(X)\displaystyle\sum_{n\geq 1}\frac{\chi(K^nX)}{n^{2g}}q^n \notag\\
 &\parallel \notag\\
 \exp\Bigl(\chi(X)\log \displaystyle&\sum_{n\geq 0}P_{g-1}(n)q^n\Bigr) =
1+\chi(X)\log \displaystyle\sum_{n\geq 0}P_{g-1}(n)q^n. \notag
\end{align}
The top equality comes from the \'etale cover $X\times K^nX\rightarrow \Hilb^nX$ of degree $n^6$, given
by the translation action of $X$ on the Hilbert scheme. 
The vertical equality is Cheah's formula (cf.~\cite{Ch}, and also~\cite{GLM} for a motivic refinement).
For the bottom equality, we treat $\chi(X)^2$ as zero when expanding $\exp$.

\subsection{Conventions}
We work over $\mathbb C$. The symbol $\chi$ denotes the topological Euler characteristic.
We denote by $\alpha\vdash n$ (one-dimensional) partitions of $n=\sum_i i\alpha_i$, corresponding to classical Young tableaux.
The number of $d$-dimensional partitions of $n$ is denoted $P_d(n)$. A higher dimensional partition can be seen
as a generalized Young tableau, with $(d+1)$-dimensional boxes taking the r{\^o}le of squares. The convention is to set $P_d(0)=1$.

\ssection{Proving the conjecture}
\subsection{Stratification}
The Hilbert scheme of points of any quasi-projective variety $X$ admits a natural stratification by partitions,
\[
\Hilb^nX=\coprod_{\alpha\vdash n}\Hilb^n_\alpha X
\]
where $\Hilb^n_\alpha X$ denotes the (locally closed) locus of subschemes of $X$ having exactly $\alpha_i$
components of length $i$.
Let $X$ be an Abelian variety. Letting $K^n_\alpha X=K^nX\cap \Hilb^n_\alpha X$, we get an induced stratification of the Kummer scheme:
\begin{equation}\label{strat}
K^nX=\coprod_{\alpha\vdash n}K^n_\alpha X.
\end{equation}
For each partition $\alpha\vdash n$, let us define the subscheme
\[
V_\alpha=\set{\xi\in \Sym^n_\alpha X|\Sigma\,\xi=0}\subset \Sym^n_\alpha X
\]
where $\Sigma$ denotes addition of zero cycles under the group law on $X$.
The Hilbert-Chow morphism $\Hilb^nX\rightarrow \Sym^nX$ restricts to morphisms
\[
\pi_\alpha:K^n_\alpha X\rightarrow V_\alpha.
\]
Fixing a point in $V_\alpha$ amounts to fixing the supporting points of the corresponding cycle and their multiplicities. Thus, 
each fibre of $\pi_\alpha$ is isomorphic to a product of punctual Hilbert schemes:
\[
F_\alpha\cong\prod_i\Hilb^i(\mathbb A^3;0)^{\alpha_i}.
\]
Hence, using \eqref{strat}, we find
\begin{equation}\label{pupupu}
\chi(K^n X)=\sum_{\alpha\vdash n}\chi(V_\alpha)\prod_iP_2(i)^{\alpha_i},
\end{equation}
where we have used $P_{d-1}(n)=\chi(\Hilb^n(\mathbb A^d;0))$ (see \cite{ES} for $d=2$ and \cite{Ch}, \cite{GLM} for the general case).

\subsection{Strategy of proof}\label{sec:strategy}

Let $\sigma_2(n) = \sum_{d|n} d^2$ denote the square sum of divisors of an integer $n$.
As is well known \cite{A}, these are related to the number of plane partitions by
\begin{equation}\label{qwer}
nP_2(n) = \sum_{k=1}^n \sigma_2(k)P_2(n-k).
\end{equation}

Let us define, for $\alpha\vdash n$, integers $c(\alpha)\in\mathbb Z$ by the recursion
\begin{equation}\label{rec}
c(\alpha)=
\begin{cases}
 n & \textrm{if}\, \alpha=(n^1),\\
 -\sum_{i, \alpha_i \ne 0} c(\hat\alpha^i) & \textrm{otherwise}.
\end{cases}  
\end{equation}
where, for a partition
$\alpha=(1^{\alpha_1}\cdots i^{\alpha_i}\cdots \ell^{\alpha_\ell}) \vdash n$,
with $\alpha_i \ne 0$, we let
\begin{equation}\label{tyuiop}
\hat\alpha^i=(1^{\alpha_1}\,\cdots\,i^{\alpha_i-1}\,\cdots\,\ell^{\alpha_\ell})\vdash n-i.
\end{equation}

We shall prove Theorem \ref{thm} in two steps, given by the two Lemmas that follow.

\begin{lemma}\label{ljl}
The square sum of divisors $\sigma_2$ can be expressed in terms of the number of
plane partitions $P_2$ as follows:
\begin{equation}\label{formulaqwe}
\displaystyle \sigma_2(n)=\sum_{\alpha\vdash n}c(\alpha)\prod_iP_2(i)^{\alpha_i}.
\end{equation}
\end{lemma}

\begin{lemma}\label{lkl}
The Euler characteristics $\chi(V_\alpha)/n^5$ equal the numbers $c(\alpha)$ defined by recursion \eqref{rec}.
\end{lemma}

Assuming the two Lemmas, the main theorem follows:

\begin{proof}[Proof of Theorem \ref{thm}]
Equation \eqref{pupupu} gives
\begin{align*}
\frac{\chi(K^nX)}{n^5}&=
\sum_{\alpha\vdash n}\frac{\chi(V_\alpha)}{n^5}\prod_iP_2(i)^{\alpha_i}\\
&=\sum_{\alpha\vdash n}c(\alpha)\prod_iP_2(i)^{\alpha_i}\\
&=\sigma_2(n).
\end{align*}
We have applied Lemma \ref{lkl} in the second equality, and Lemma \ref{ljl} in the last equality.
\end{proof}

\subsection{Proof of Lemma \ref{ljl}: a recursion}

Let us introduce the shorthand
\begin{equation*}
f(\alpha) = \prod_i P_2(i)^{\alpha_i}.
\end{equation*}
Expand the right hand side of \eqref{formulaqwe}, using the definition of $c(\alpha)$:
\begin{equation}\label{eq:doublesum1}
\sum_{\alpha\vdash n} c(\alpha) f(\alpha)
= nP_2(n) -
\sum_{\substack{\alpha\vdash n\\ \alpha \ne (n^1)}}
\sum_{\substack{j\ge 1\\ \alpha_j\ne 0}} c(\hat\alpha^j) f(\hat\alpha^j)
\end{equation}
On the other hand, by induction on $n$, the identity \eqref{qwer} gives
\begin{equation}\label{eq:doublesum2}
\sigma_2(n) = n P_2(n) - \sum_{k=1}^{n-1} \sigma_2(k) P_2(n-k)
= n P_2(n) - \sum_{k=1}^{n-1} \sum_{\beta\vdash k} c(\beta) f(\beta) P_2(n-k).
\end{equation}
The sets over which the double sums in \eqref{eq:doublesum1} and
\eqref{eq:doublesum2} run are clearly identified via $(k, \beta) = (n-j,
\hat\alpha^j)$. Since $f(\alpha) = P_2(j) f(\hat\alpha^j)$, it follows that the
two expressions \eqref{eq:doublesum1} and \eqref{eq:doublesum2} are identical.
Lemma \ref{ljl} is established.

\subsection{Proof of Lemma \ref{lkl}: an incidence correspondence}
In this section we prove Lemma \ref{lkl}.
The technique used is very similar to the one adopted in~\cite{G1}.

Later on, we will need the following:
\begin{rem}\label{opkl}
Let $\alpha=(n^1)$. Then $V_\alpha$ is in bijection with the subgroup $X_n\subset X$ of $n$-torsion points in $X$.
This implies that $\chi(V_\alpha)=\chi(X_n)=n^6$. In other words, $\chi(V_\alpha)/n^5=n=c(\alpha)$.
\end{rem}

Now we fix a partition $\alpha\vdash n$ different from $(n^1)$, and 
an index $i$ such that $\alpha_i\neq 0$. We will compute $\chi(V_\alpha)$ in terms of 
the partition $\hat\alpha^i\vdash n-i$, thanks to an incidence correspondence between the spaces
$V_\alpha\subset \Sym^n_\alpha X$ and $V_{\hat\alpha^i}\subset \Sym^{n-i}_{\hat\alpha^i} X$. 

\medskip
Let us define the subscheme
\[
I=\set{(a,b;\xi)\in X^2\times V_\alpha|\mult_a\xi=i,\,(n-i)b=ia \textrm{ in }X}\subset X^2\times 
V_\alpha.
\]
We use the incidence correspondence 
\[
\begin{tikzcd}
I\arrow{r}{\phi}\arrow[swap]{d}{\psi} & V_\alpha \\
V_{\hat\alpha^i}
\end{tikzcd}
\]
where the map $\phi$ is the one induced by the second projection, and
$\psi$ sends $(a,b;\xi)$ to the cycle $T_b(\xi-ia)$, where $T_b$ is translation by $b\in X$.

\medskip
The strategy is to compute $\chi(I)$ twice: by means of the fibres of $\phi$ and $\psi$ respectively.
This will enable us to compare $\chi(V_\alpha)$ and $\chi(V_{\hat\alpha^i})$.

\medskip

\paragraph{Fibres of $\phi$.}
Let $\xi\in V_\alpha$. This means $\xi\in \Sym^n_\alpha X$ and $\sum \xi=0$ in $X$. We have
\[
\phi^{-1}(\xi)=\set{(a,b)\in X^2|\mult_a\xi=i,\,(n-i)b=ia}\subset X^2.
\]
Let $a_1,\dots,a_{\alpha_i}$ be the $\alpha_i$ points, in the support of $\xi$, having multiplicity $i$ (recall that $i$ is fixed).
Then 
\[
\phi^{-1}(\xi)=\coprod_{1\leq j\leq \alpha_i}H_j,
\]
where $H_j=\set{b\in X|(n-i)b=ia_j}$.
Each $H_j$ is the kernel of the translated isogeny $b\mapsto (n-i)b-ia_j$, which has degree $(n-i)^6$, 
so $\chi(H_j)=(n-i)^6$. This yields $\chi(\phi^{-1}(\xi))=\alpha_i(n-i)^6$. Hence,
\begin{equation}\label{suca11}
\chi(I)=\chi(V_\alpha)\alpha_i(n-i)^6.
\end{equation}

\paragraph{Fibres of $\psi$.}

Let $C\in V_{\hat\alpha^i}$. A point $(a,b;\xi)\in \psi^{-1}(C)$ determines $\xi$ as
\[
\xi=T_b^{-1}(C)+ia,
\]
and the condition $\mult_a\xi=i$ translates into
$\mult_a(T_b^{-1}(C)+ia)=i$, which means $a\notin \Supp(T_b^{-1}(C))$, i.e.~$a+b\notin \Supp(C)$.

\medskip
Let us define the subscheme 
\[
B=\set{(a,b)|(n-i)b=ia}\subset X^2.
\]
Then we note that
\[
\psi^{-1}(C)=\set{(a,b)\in B|a+b\notin \Supp(C)}=B\setminus \coprod_{c\in \Supp(C)}Y_c,
\]
where 
\[
Y_c=\set{(a,b)\in B|a+b=c}\cong \set{b\in X|nb=ic}\cong X_n.
\]

Now, if we map $B\rightarrow X$ through the second projection, we see that the fibres are all isomorphic (to $X_i$, the 
group of $i$-torsion points in $X$). Hence, as $\chi(X)=0$, we find that $\chi(B)=0$. 
Thus, remembering that $\Supp(C)$ consists of $(\sum_i\alpha_i)-1$ distinct points, we find
\[
\chi(\psi^{-1}(C))=-\sum_{c\in \Supp(C)}\chi(Y_c)=-n^6\cdot\Bigl(\sum_i\alpha_i-1\Bigr).
\]
Finally,
\begin{equation}\label{suca22}
\chi(I)=-\chi(V_{\hat\alpha^i})n^6\cdot\Bigl(\sum_i\alpha_i-1\Bigr).
\end{equation}

\medskip
Compare \eqref{suca11} and \eqref{suca22} to get
\[
\chi(V_{\hat\alpha^i})=-\frac{\alpha_i(n-i)^6}{n^6\bigl(\sum_i\alpha_i-1\bigr)}\chi(V_\alpha).
\]

We now conclude by showing that the numbers $\chi(V_\alpha)/n^5$
satisfy the same recursion \eqref{rec} fulfilled by the $c(\alpha)$'s. 
If $\alpha=(n^1)$, we know by Remark $\ref{opkl}$ that 
\begin{equation*}
\frac{1}{n^5}\chi(V_\alpha)=n. 
\end{equation*}
For $\alpha \ne (n^1)$, we can use the above computations to find (the sums run over all
indices $i$ for which $\alpha_i\ne 0$):
\begin{align*}
-\sum_i \frac{1}{(n-i)^5}\chi(V_{\hat\alpha^i})&=\sum_i\frac{1}{(n-i)^5}\frac{\alpha_i(n-i)^6}
{n^6\cdot\bigl(\sum_i\alpha_i-1\bigr)}\chi(V_\alpha)\\
&=\frac{1}{n^5}\frac{\sum_i\alpha_i(n-i)}{n\bigl(\sum_i\alpha_i-1\bigr)}\chi(V_\alpha)\\
&=\frac{1}{n^5}\frac{n\sum_i\alpha_i-\sum_ii\alpha_i}{n\sum_i\alpha_i-n}\chi(V_\alpha)\\
&=\frac{1}{n^5}\chi(V_\alpha).
\end{align*}
Lemma \ref{lkl} is proved. As noted in
Section \ref{sec:strategy}, this completes the proof of Theorem \ref{thm}.

\begin{rem}
For an Abelian variety $X$ of arbitrary dimension $g$, Shen \cite{shen} observes that from
an equality of formal power series in $q$,
\begin{equation*}\label{identtt}
\sum_{n\geq 0}P_{g-1}(n)q^n=\exp\Bigl(\sum_{n\geq 1}s_nq^n \Bigr),
\end{equation*}
defining the sequence $\{s_n\}_{n\ge 1}$,
one obtains by application of the operator $q\frac{\textrm d}{\textrm dq}$
the identity
\begin{equation*}
nP_{g-1}(n) = \sum_{k=1}^n ks_kP_{g-1}(n-k).
\end{equation*}
Starting with this equality, our proofs of Lemmas \ref{ljl}
and \ref{lkl}, with $\chi(V_{\alpha})/n^5$ replaced by $\chi(V_{\alpha})/n^{2g-1}$,
go through without change, and we recover the identity \eqref{eq:chi-series}.
\end{rem}


\begin{thebibliography}{9}

\bibitem{A}
G. E. Andrews, \emph{The theory of partitions},
Cambridge University Press, Cambridge (1998).

\bibitem{Ch}
J. Cheah, \emph{On the cohomology of Hilbert schemes of points}, 
J. Algebraic Geom. 5 (1996), no. 3, 479–511.

\bibitem{D}
O. Debarre, \emph{On the Euler Characteristic of Generalized Kummer Varieties},
American Journal of Mathematics Vol. 121, No. 3 (Jun., 1999), pp. 577-586.

\bibitem{ES}
G. Ellingsrud, S. A. Str{\o}mme, \emph{On the homology of the Hilbert scheme of points in the plane},
Invent. Math. 87, 343-352 (1987).

 
\bibitem{GLM}
S. M. Gusein-Zade, I. Luengo, and A. Melle-Hernández, 
\emph{Power structure over the Grothendieck ring of varieties and generating series of Hilbert schemes of points},
Michigan Math. J. Volume 54, Issue 2 (2006), 353-359.

\bibitem{G1}
M. G. Gulbrandsen, \emph{Computing the Euler characteristic of generalized Kummer varieties},
Arkiv för Matematik April 2007, Volume 45, Issue 1, pp 49-60.

\bibitem{G2}
M. G. Gulbrandsen, \emph{Donaldson-Thomas invariants for complexes on abelian threefolds},
Mathematische Zeitschrift, Volume 273, Issue 1-2 (February 2013), pp 219-236.

\bibitem{shen}
J. Shen, \emph{The Euler characteristics of generalized Kummer schemes},
arXiv:1502.03973.

\bibitem{St}
R. P. Stanley, \emph{Enumerative Combinatorics}, Volume 2, Cambridge University Press 1999.


\end{thebibliography}
\end{document}